\documentclass[secnum,leqno]{rims-bessatsu}%%% eqno resets every section 
\usepackage{amssymb, amsmath}%%% option packages
\usepackage{graphics}%%% option packages
\usepackage{graphicx}%%% option packages
\usepackage{pict2e}
\usepackage{eepic}%%% option packages
\usepackage{epic}%%% option packages
\newtheorem{theorem}{Theorem}[section]

\theoremstyle{definition}

\theoremstyle{remark}
\newtheorem*{remark}{Remark}%%% * means no numbering
\title{Survey and remarks on Viro's definition \\ of Khovanov homology}
\TitleHead{Survey and remarks on Viro's definition of Khovanov homology}          %optional
\author{\textsc{Noboru Ito}\footnote{Department of Mathematics, Waseda University, Tokyo 169-8555, Japan.\newline e-mail: \texttt{noboru@moegi.waseda.jp}}}
\AuthorHead{Noboru Ito}         %optional but recommended
\support{This work was supported by Grant-in-Aid for Scientific Research (23740062).}  %optional
\VolumeNo{x}           %editors must define this.
\YearNo{201x}           %editors must define this.
\PagesNo{000--000}      %editors must define this.
\communication{}      %editors must define this.
\begin{document}
%
% The text goes here.  
% Be sure to use the appropriate "theorem-like" environment as 
% is the following examples.  Never use plain TeX commands for these, as
% they will cause interference with the styles of other papers. 

\maketitle

%\tableofcontents      %optional
\begin{abstract}      %optional
This paper reviews and offers remarks upon Viro's definition of the Khovanov homology of the Kauffman bracket of unoriented framed tangles (Sec. \ref{def_kh}).  The review is based on a file of his talk.  This definition contains an exposition of the relation between the $R$-matrix and the Kauffman bracket (Sec. \ref{unoriented}).  
\end{abstract}

\section{Introduction}
The Khovanov homology, in which bigraded groups are link invariants, was introduced by Khovanov \cite{khovanov0}, and ever since, much research effort has been developed to exploring this homology.  Today, we can find several interpretations of the Khovanov homology.  Here, we will explore Viro's definition of the framed Khovanov homology in \cite{v} and \cite{v2}.  We have chosen this because the file \cite{v2} of Viro's talk maintains that there is a direct correspondence among generators of the framed Khovanov homology and the element of the $R$-matrix of the Kauffman bracket, a view that has not been seen elsewhere.  As is well-known, the $R$-matrix of the Jones polynomial is made available by a normalization of that of the Kauffman bracket.  

The outline of this paper is as follows.  Sec. \ref{def_kh} is concerned with Viro's definition of the framed Khovanov homology with regard to the Kauffman bracket.  In Sec. \ref{def_kh_l}, we define the framed Khovanov homology for the Kauffman bracket of framed links.  In Sec. \ref{unoriented}, we extend the definition of the framed Khovanov homology to that of unoriented framed tangles via the following three steps.  First, we generalize the generators of a free abelian group, which become complex and are suited to tangles (Sec. \ref{generalized_state}).  Second, we define homological grading and $A$-grading for the complex (Sec. \ref{grading}).  Third, we define a boundary operator $\partial$ for the complex, and with this, complete the definition of the framed Khovanov homology (Sec. \ref{boundary}).  We explain how we extend the proofs that $\partial^{2}$ $=$ $0$ and the homotopy equivalence of this framed Khovanov homology, which implies invariance under an isotopy of tangles using only elementary techniques.  In Sec. \ref{r_matrix}, we observe the relation between the $R$-matrix and the framed Khovanov homology as defined in Secs. \ref{def_kh_l} and \ref{unoriented}.

\section{Viro's definition of framed Khovanov homology of Kauffman bracket}\label{def_kh}
\subsection{Khovanov homology for framed links.}\label{def_kh_l}
In this section, we recall the definition of the {\it{framed Khovanov homology}} of the Kauffman bracket given by Viro \cite{v}.  Many definitions related to the Khovanov complex are quoted from \cite{v} (e.g., Secs. 4.4, 5.4.C, 6.1).  Now let us consider a framed link diagram $D$ by means of blackboard framing.  By the abuse of notation, when there is no danger of confusion, we write a link diagram to denote a framed link diagram.  To begin with, as written in Fig. \ref{mark_sth_mark}, we place a small edge, called a {\it{marker}}, for every crossing on the link diagram.  In the rest of this paper, we can use the simple notation as in Fig. \ref{mark_sth_mark}.  Every marker has its sign, as defined in Fig. \ref{mark_sth_mark}.  That is, we make Kauffman states using these signed markers, smoothing all the crossings of a link diagram along all the markers\footnote{Note that, in Viro's paper \cite{v}, the (Kauffman) state is a distribution of markers.  } (Fig. \ref{mark_sth}).
\begin{figure}
\qquad\qquad
\includegraphics[width=10cm
]{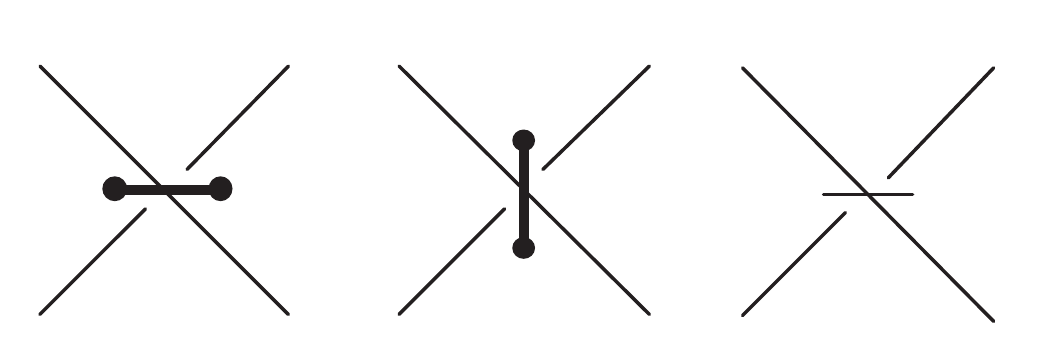}
\caption{A positive marker, a negative marker, and a simple notation for a positive marker from the left-hand side.}\label{mark_sth_mark}
\end{figure}
\begin{figure}
\qquad\qquad
\includegraphics[width=10cm]{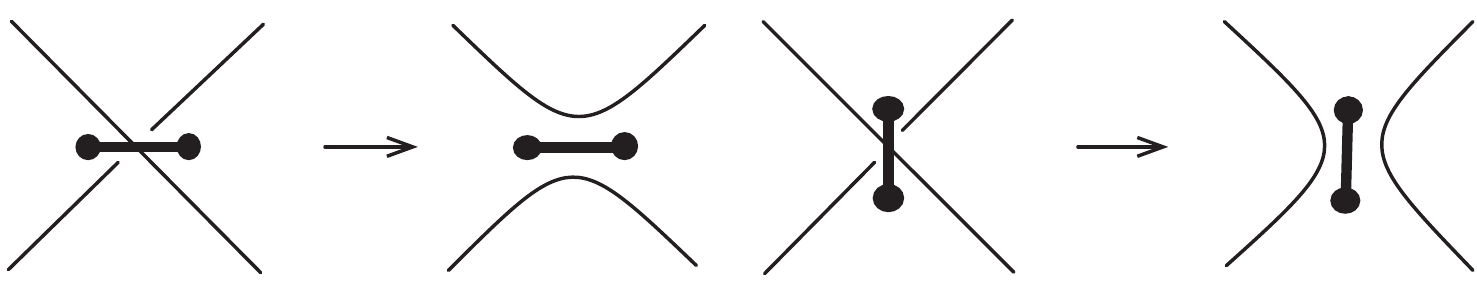}
\caption{Smoothing producing Kauffman states.  The marker on the crossing in the left figure is the positive marker; that in the right figure is the negative marker.}\label{mark_sth}
\end{figure}
We give an order of negative markers of a Kauffman state, which are considered up to an even permutation.  We say that orientations are opposite (resp. same) if they differ by an odd (resp. even) permutation for two orders of negative markers.  Now, we consider the relation among the Kauffman states such that one state is equal to another state multiplied by $-1$ (resp. $1$) if two Kauffman states have opposite (resp. same) orientations \cite[Sec. 5.4.C]{v}.  Next, we put an $x$ or $1$ for each circle of the Kauffman state.  Lets the degree of $x$ be $-1$ and that of $1$ be $1$, following \cite{v2}\footnote{In \cite{v}, Viro uses signs that replace $+$ (resp. $-$) with $1$ (resp. $x$)}.  We denote degrees of $x$ or $1$ by ${\rm{deg}}(x)$ or ${\rm{deg}}(1)$.  The Kauffman state whose circles have $x$ or $1$ with the relation defined by the orientation of the negative makers of the Kauffman state is called the {\it{enhanced Kauffman state}}, as denoted by $S$.  Let $\sigma(S)$ be the sum of the signs of the markers in $S$ and $\tau(S)$ $=$ $\sum_{{\text{circles}}\,y\,{\text{in}}\,S} {\mathrm{deg}} (y)$.  Here ${\rm{deg}}(y)$ $=$ ${\rm{deg}}(x)$ or ${\rm{deg}}(1)$, that is, $-1$ or $1$.  Following \cite[Section 6.1]{v}, we obtain
\begin{equation}
\langle D \rangle = \sum_{{\text{enhanced states}}\,S\,{\text{of}}\,D} (-1)^{\tau(S)} A^{\sigma(S) - 2 \tau(S)}, 
\end{equation}
where $\langle D \rangle$ denotes the Kauffman bracket for $D$.  Here we note that $\langle$unknot with $0$-framing$\rangle$ $=$ $-A^{2}$ $-A^{-2}$.  Let $p(S)$ $=$ $\tau(S)$ and $q(S)$ $=$ $\sigma(S)$ $-$ $2\tau(S)$.  Denote by ${\mathcal{C}}_{p, q}(D)$ the free abelian group generated by the enhanced Kauffman states $S$ of a link diagram $D$ with $p(S)$ $=$ $p$ and $q(S)$ $=$ $q$ which is quotiented by the relation such that the same enhanced Kauffman states with opposite orientations differ by multiplication by $-1$ \cite[Sec. 4.4]{v}.  Let $T$ be an enhanced Kauffman state given by replacing a neighborhood of only one crossing with a positive marker, with that of a negative marker, if the neighborhood in each of the cases is as listed in Fig. \ref{frobenius_fig}.  For an enhanced Kauffman state $S$, $\partial(S)$ is defined by 
\begin{equation}\label{def_diff}
\partial(S) = \sum_{{\text{oriented enhanced Kauffman states}}~T} (S : T)\,T
\end{equation}
where the number $(S : T)$ is $1$ in each of the cases listed in Fig. \ref{frobenius_fig} if the orientations of the negative markers of $S$ and $T$ coincide on the common markers.  This is followed by the changing the crossing in the ordering for $T$ \cite[Sec. 5.4.C]{v}, and $(S : T)$ is $0$ if $S$ does not appear to the left of the arrows in Fig. \ref{frobenius_fig}.  
\begin{figure}
\qquad\qquad
\includegraphics[width=10cm]{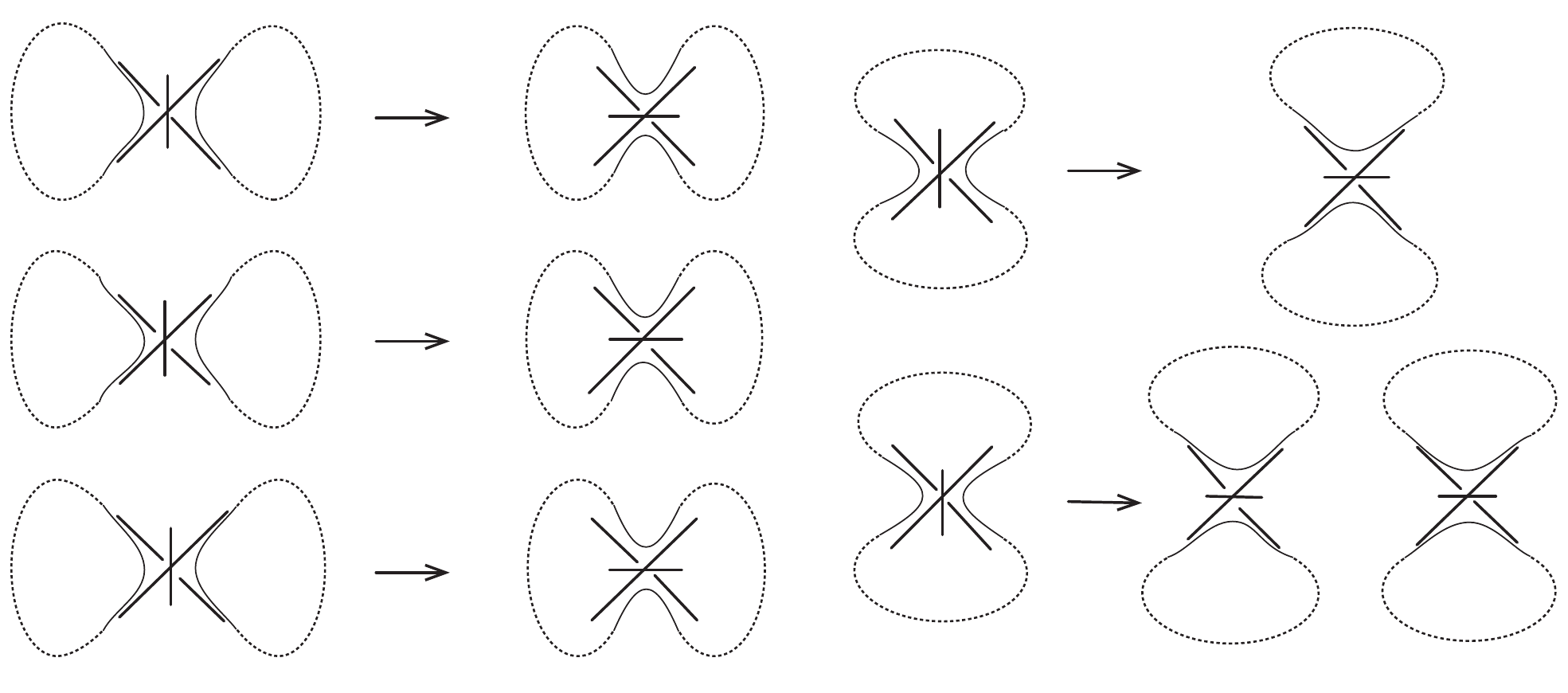}
\begin{picture}(0,10)
\put(-260,123){$S$}
\put(-173,123){$T$}
\put(-120,123){$S$}
\put(-45,123){$T$}
\put(-278,18){$1$}
\put(-278,61){$1$}
\put(-278,103){$x$}
\put(-242,18){$1$}
\put(-242,61){$x$}
\put(-242,103){$1$}
\put(-185,103){$x$}
\put(-185,61){$x$}
\put(-185,18){$1$}
\put(-120,45){$1$}
\put(-120,105){$x$}
\put(-45,108){$x$}
\put(-45,72){$x$}
\put(-67,48){$x$}
\put(-66,14){$1$}
\put(-25,48){$1$}
\put(-24,15){$x$}
\put(-46,32.5){$+$}
\end{picture}
\caption{Each of the figures to the left of the arrows is $S$, those to the right are $T$ for (\ref{def_diff}).  The arrow in the lower left implies that $S$ produces two types of $T$.}\label{frobenius_fig}
\end{figure}
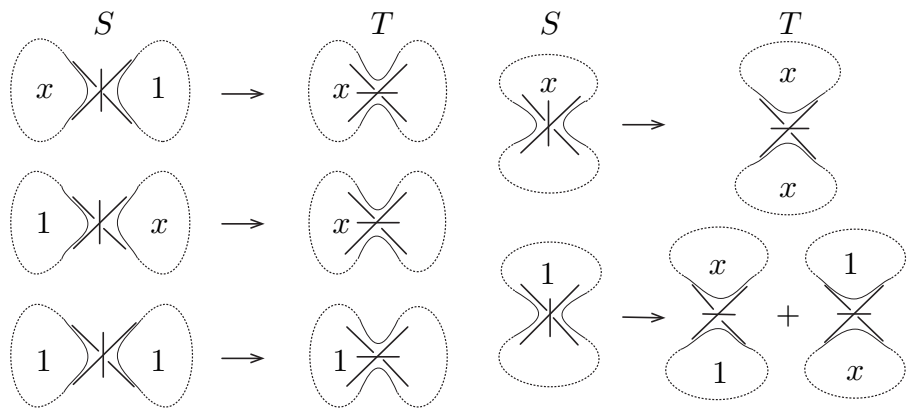
The map for enhanced Kauffman states is extended to the homomorphism $\partial : {\mathcal{C}}_{p, q}$ $\to$ $\mathcal{C}_{p - 1, q}$.  The homomorphism $\partial$ satisfies $\partial^{2}$ $=$ $0$ (this fact is non-trivial, see \cite[Sec. 5.4.D]{v}) and then the $\partial$ becomes a boundary operator.  The {\it{framed Khovanov complex of the Kauffman bracket}} is defined as the complex $\{{\mathcal{C}}_{p, q}(D), \partial \}$.  The {\it{framed Khovanov homology}} defined by this complex is denoted by $H_{*, *}$.  For this framed Khovanov homology, we have
\begin{equation}\label{euler_ch}
\begin{split}
\langle D \rangle &= \sum_{q} A^{q} \sum_{p} (-1)^{p} {\mathrm{rank}}\,{\mathcal{C}}_{p, q}(D)\\
&= \sum_{q} A^{q} \sum_{p} (-1)^{p} {\mathrm{rank}}\,H_{p, q}(D).  
\end{split} 
\end{equation}

\begin{theorem}[Khovanov]
Let $D$ be a framed link diagram of an unoriented framed link $L$.  Then the homology $H_{*}({\mathcal{C}}_{*, *}(D), \partial)$ is an invariant of $L$.  
\end{theorem}
\begin{proof}
The proof is essentially the same computation as for the case of unframed links (e.g. \cite{ito}).  
\end{proof}

\subsection{Khovanov homology for unoriented framed tangles.}\label{unoriented}

\subsubsection{Generalized states.}\label{generalized_state}
In this section, we will extend the enhanced Kauffman states of links to those of tangles \cite{v2}.  First, for all the crossings of a given tangle, we assign markers, as described in Sec. \ref{def_kh}.  After we smoothen all the crossings along all the markers, we have a set of arcs or circles (Fig. \ref{general_state}), which is also called the {\it{Kauffman state}} $s$ here.  
\begin{figure}
\qquad\qquad
\includegraphics[width=10cm]{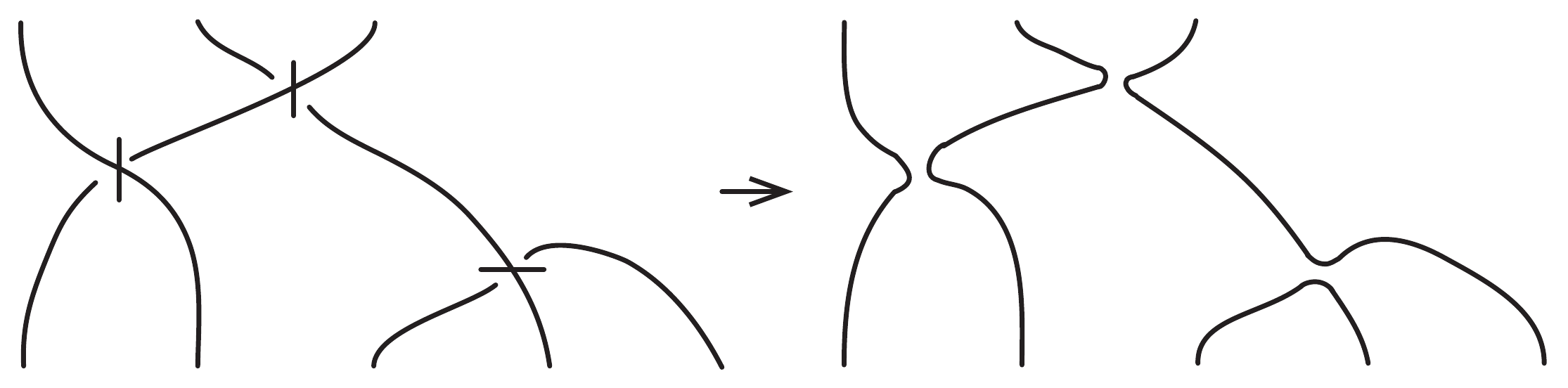}
\caption{Example of the Kauffman state $s$ obtained by smoothing along markers.}\label{general_state}
\end{figure}
\begin{figure}
\qquad\qquad
\includegraphics[width=10cm]{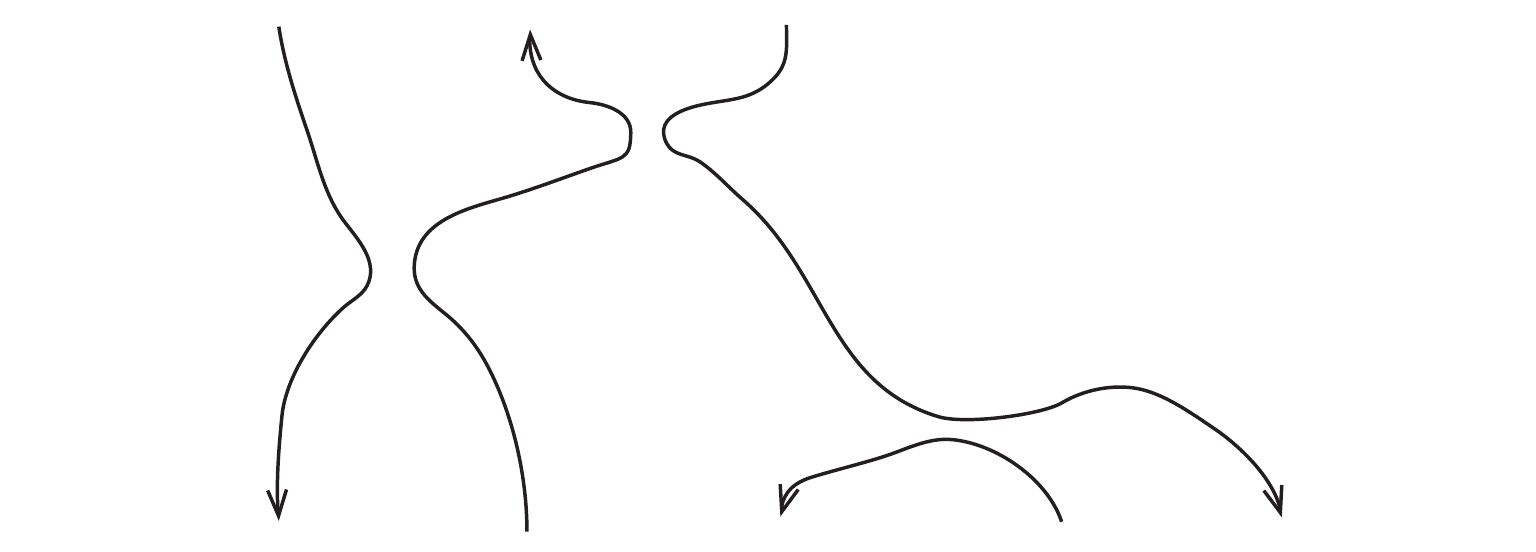}
\caption{Example of the enhanced Kauffman state $S$ corresponding to Fig. \ref{general_state}.}\label{enhanced_state}
\end{figure}
We assign an arbitrary orientation to each arc of the Kauffman state $s$ considering any possibility.  The Kauffman state whose arcs and circles have orientations is also called the {\it{enhanced Kauffman state}} $S$ here (Fig. \ref{enhanced_state}).  If we intuitively consider the case of the links, we see that $x$ and $1$ correspond to two ways of assigning orientations.  

\subsubsection{Homology grading and $A$-grading.}\label{grading}
Let us redefine $p(S)$ $=$ $\tau(S)$.  Let $\tau(S)$ be the degree of the Gauss map of $S$ be evaluated as the average of the local degree at $\pm 1 \in S^{1}$.  In other words, $\tau(S)$ is the sum of the local degree as defined in Fig. \ref{local_degree}.  
\begin{figure}
\qquad\qquad
\includegraphics[width=10cm]{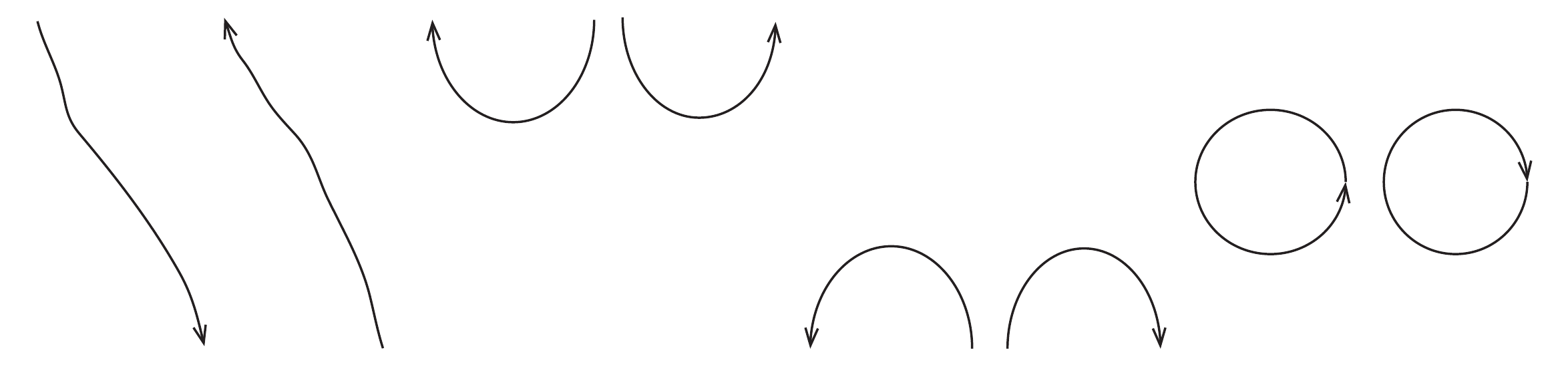}
\caption{Local degrees of Gauss maps.  From the left, the degree is $0$, $0$, $- \frac{1}{2}$, $\frac{1}{2}$, $\frac{1}{2}$, $- \frac{1}{2}$, $1$, and $- 1$.  The degrees $-1$ and $1$ correspond to those of $x$ and $1$.}\label{local_degree}
\end{figure}
The $A$-grading $q(S)$ is defined by the same formula $q(S)$ $=$ $\sigma(S)$ $-$ $2\tau(S)$ as in Sec. \ref{def_kh_l}.  By using ordering markers as in Sec. \ref{def_kh_l}, the $\mathbb{Z}$-module generated by the enhanced Kauffman states of tangles satisfies $p(S)$ $=$ $p$ and $q(S)$ $=$ $q$ and we denote the bigraded module by $\mathcal{C}_{p, q}(D)$ for an arbitrary framed tangle diagram $D$.  

\subsubsection{The Kauffman bracket of unoriented framed tangles.}
Using Secs. \ref{generalized_state} and \ref{grading}, we obtain the Kauffman bracket of an unoriented framed tangle $D$ by the following formula: 
\begin{equation}
\langle D \rangle = \sum_{{\text{generalized states}}\,S\,{\text{of}}\,D} (-1)^{\tau(S)} A^{\sigma(S) - 2 \tau(S)}, 
\end{equation}
where $\langle D \rangle$ denotes the Kauffman bracket for $D$.  Here we note that $\langle$$0$-framed simple arc with no local maximums or minimums$\rangle$ $=$ $2$, $\langle$$0$-framed arc with either local maximum or minimum$\rangle$ $=$ $-(-1)^{\frac{1}{2}}$$A$ $+$ $(-1)^{\frac{1}{2}}$$A^{-1}$, and $\langle$$0$-framed unknot$\rangle$ $=$ $-A^{2}$ $-A^{-2}$ by the definition (Fig. \ref{elem_tangle_kauffman}).  
\begin{figure}
\qquad\qquad
\includegraphics[width=10cm]{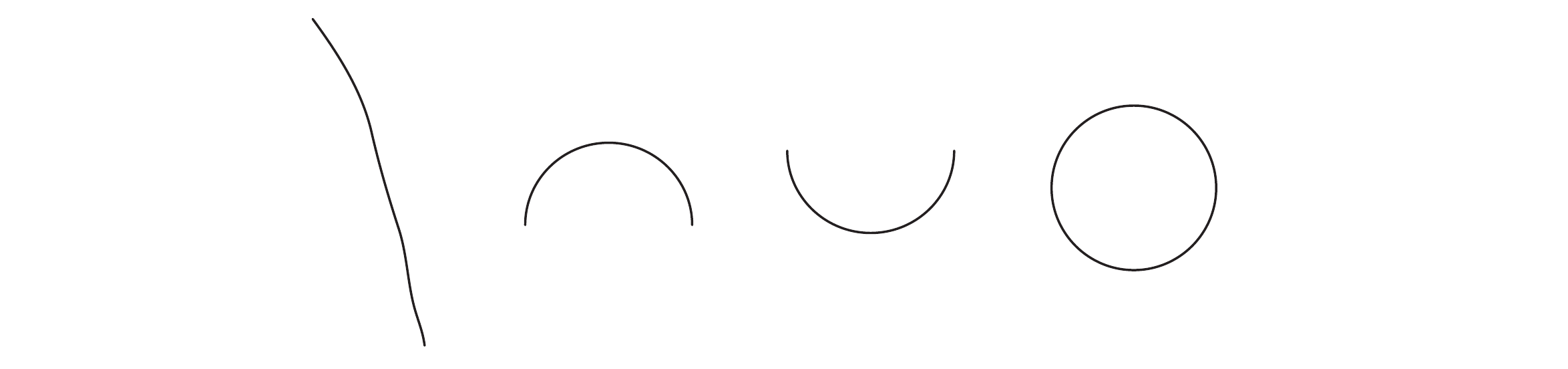}
\caption{Elementary framed tangles.  From the left-hand side, a $0$-framed simple arc with no local maximums or minimums, a $0$-framed arc with a local maximum, a $0$-framed arc with a local minimum, a $0$-framed unknot.  }\label{elem_tangle_kauffman}
\end{figure}

\subsubsection{Boundary operator for tangles.}\label{boundary}
Let us consider the map between enhanced Kauffman states which changes a single positive marker to a negative marker, changes adjacent orientations at a new negative marker, and satisfies the following conditions: 
\begin{itemize}
\item the $A$-grading would be preserved, 
\item the homology grading would decrease by $1$, and
\item the orientations at the end points would be preserved.  
\end{itemize}
If we extend the map linearly, we have a map on $\mathcal{C}_{p, q}(D)$ of a given framed tangle $D$ denoted by $\partial$ (this definition is a natural generalization of the case of links, so we use the same terminology and symbols as in Sec. \ref{def_kh_l}).  
As written in \cite{v2}, the following Theorem is held.  
\begin{theorem}[Viro]
\begin{equation}\label{differential}
{\partial}^{2} = 0.  
\end{equation}
\end{theorem}
\begin{proof}
We can easily extend the proof of the case of links and the coefficient $\mathbb{Z}_2$ \cite[Proof of Theorem 5.3.A]{v} to that of framed tangles and the coefficient $\mathbb{Z}$ by checking equation (\ref{differential}) along the figures in \cite[Pages 333, 334]{v} as follows: 
\begin{itemize}\label{steps}
\item Replace $+$ and $-$ with $1$ and $x$, respectively; 
\item Give circles with $x$ (resp. $1$) orientations whose degree of Gauss map is $- 1$ (resp. $1$), 
\item to make the enhanced Kauffman state with open arcs, put a point on the circle, except on every crossing of the link that corresponds to an open arc, and 
\item the image $\partial(S)$ for an enhanced Kauffman state $S$ is $0$ if the sign of a pointed circle (= open arc) must be changed,   
\end{itemize}
and where there is ambiguity in the choice of an edge of a circle, to put a point (= to cut the closed circle at the point) that is selected in the second step, we have to consider every possibility.  Following this the modification to the process is done above, we note that each pointed circle takes $0$ in one route if and only if the pointed circle takes $0$ in another route in each diagram of \cite[Pages 333, 334]{v}.  Therefore, in the case of tangles as well, the figures in \cite[Pages 333, 334]{v} with base points are held where some enhanced Kauffman states may be $0$, depending upon which are open arcs.  
\end{proof}
Next, Theorem \ref{kh_tan} is a generalization of the Khovanov homology of framed links (Sec. \ref{def_kh_l}).  
\begin{theorem}[Viro]\label{kh_tan}
Let $D$ be a diagram of an arbitrary unoriented framed tangle.  The homology $H_{*}(\mathcal{C}_{*, q}(D), \partial)$ is an isotopy invariant of unoriented framed tangles.  
\end{theorem}
\begin{proof}
By definition, a proof \cite[Secs. 2.2, 2.3]{ito} of the Khovanov homology of links can be interpreted as that of unoriented framed links.  Therefore, here, we extend the proof \cite[Secs. 2.2, 2.3]{ito} of the invariance of the Khovanov homology of unoriented framed link diagrams of the second and third Reidemeister moves to that of unoriented framed tangles.  To apply the proof \cite[Secs. 2.2, 2.3]{ito}, replace $+$ and $-$ with $x$ and $1$, respectively (note that the definitions of signs in \cite{ito} are exchanged for those in \cite{v2}).  Next, the latter three steps of (\ref{steps}) are performed.  What we have to show is the two equalities (Step 1) $h \circ \partial$ $+$ $\partial \circ h$ $=$ $\operatorname{id}$ $-$ $\operatorname{in} \circ \rho$ ($\ast$) and (Step 2) $\partial \circ \rho$ $=$ $\rho \circ \partial$ ($\ast\ast$), where $h$ and $\rho$ are defined as in \cite[Secs. 2.2, 2.3]{ito}, ``$\operatorname{id}$'' is the identity map, and ``$\operatorname{in}$'' is the inclusion map.

(Step 1) To begin with, we can readily notice the following property for $h$.  For any enhanced Kauffman state $S$ that has open arcs, $\partial(S)$ is $0$ (resp: $S$ $=$ $0$), because $S$ has an arc-fixed orientation if and only if $h \circ \partial(S)$ $=$ $0$ (resp: $h(S)$ $=$ $0$) for the same reason.  

(Case 1 of Step 1) Next, we consider the case of $S$ such that the definition of $\rho(S)$ is not concerned with two times of Frobenius calculuses $(p:q):(q:p)$ and $(q:p):(p:q)$.  In this case, if every component of the enhanced Kauffman state $S$ is a circle, of course, the equality ($\ast$) is certainly held, since we have already shown this equality in \cite[Secs. 2.2, 2.3]{ito}.  Then, if the enhanced Kauffman states appear in $(\partial \circ h$ $+$ $h \circ \partial)(S)$, all the states appearing on the left hand side ($\ast$) have to appear in $(\operatorname{id}$ $-$ $\operatorname{in} \circ \rho)(S)$ and the inverse is true.  Moreover, in this case, by this assumption, a single Frobenius calculus determines which of each $S$ survives or not.  In other words, the condition of the survivability depends upon a single Frobenius calculus for both sides.  Then, the equality ($\ast$) is still held in this case.  

(Case 2 of Step 1) Now we consider another case of $S$ such that the definition of $\rho(S)$ is concerned with two times of the Frobenius calculus $(p:q):(q:p)$ and $(q:p):(p:q)$.  If the case $(p:q):(q:p)$ $=$ $0$ or $(q:p):(p:q)$ $=$ $0$, the discussion returns to Case 1.  If $(p:q)$ and $(q:p)$ must survive, the case is similar to Case 1, and the condition of survivability depends on a single Frobenius calculus, and the discussion returns to Case 1.  Consider the case that satisfies $(p:q)$ and $(q:p)$ $=$ $0$, because the orientation of some arcs are fixed if and only if $S$ with $(p:q):(q:p)$ and $(q:p):(p:q)$ $=$ $0$ for the same reason.  In this case, if we take $S$ with $(p:q)$ and $(q:p)$ $=$ $0$, again, the condition of survivability depends on a single Frobenius calculus, and the discussion returns to Case 1, and if $S$ with $(p:q)$ and $(q:p)$ $\neq$ $0$, this is the same situation as that of the links.  We check every possibility of the case concerned with $(p:q):(q:p)$ and $(q:p):(p:q)$; all above cases were checked above in this manner.  

(Step 2) What must be done is to check all the cases, but some key points are  represented below.  

As we see in (Case 2) of (Step 1), for an enhanced Kauffman state $S$, two Frobenius calculuses $(p:q):(q:p)$ and $(q:p):(p:q)$ produced by a single Frobenius calculuses $(p:q)$ and $(q:p)$ concerned with the definition of $\rho(S)$ is survived only if $S$ with $(p:q)$ and $(q:p)$ survives.  We can check ($\ast\ast$) easily in the case of the second Reidemeister move.  For the second Reidemeister move, except for the case of the enhanced Kauffman state with two positive markers that concerns the move (to show the case, we repeatedly use the formulae of definitions of $\rho$), we use the rule ``multiplying the unit'' of the Frobenius calculus which is concerned with a circle with $1$.  For the third Reidemeister move, except for the case in which the markers are positive, negative, and positive for $a$, $b$, and $c$, respectively, of the picture in \cite[Section 2.2, Formula 2.4]{ito}, we use the rules that changing a marker and the behavior of a circle with $1$ is that of the unit.  The left-behind case is similar to (Case 2) of (Step 1).  The two Frobenius calculuses concerned with the non-zero enhanced Kauffman bracket $S$ with $(p:q):(q:p)$ and $(q:p):(p:q)$ appearing in the image $\rho(S)$ is well-defined only if the non-zero Kauffman bracket $S$ with $(p:q)$ and $(q:p)$ is well-defined.  Then, the discussion of two Frobenius calculuses is reduced to one Frobenius calculus, as in (Case 1) of (Step 1).  At first glance, the discussion of ($\ast$) is still about two Frobenius calculuses, but the discussion is reduced to that of one Frobenius calculus in the end.  

Here, this proof has been completed.  
\end{proof}
\begin{remark}
A proof of the invariance of the Khovanov homology of links in a systematic context will be given elsewhere.  
\end{remark}

\subsubsection{$R$-matrix and enhanced Kauffman states}\label{r_matrix}
In this section, we comment upon the relation $R$-matrix and enhanced Kauffman states.  $A$-grading can be localized, as in Fig. \ref{bolz_w} (see also \cite{v2}).  
\begin{figure}
\qquad\qquad
\includegraphics[width=10cm]{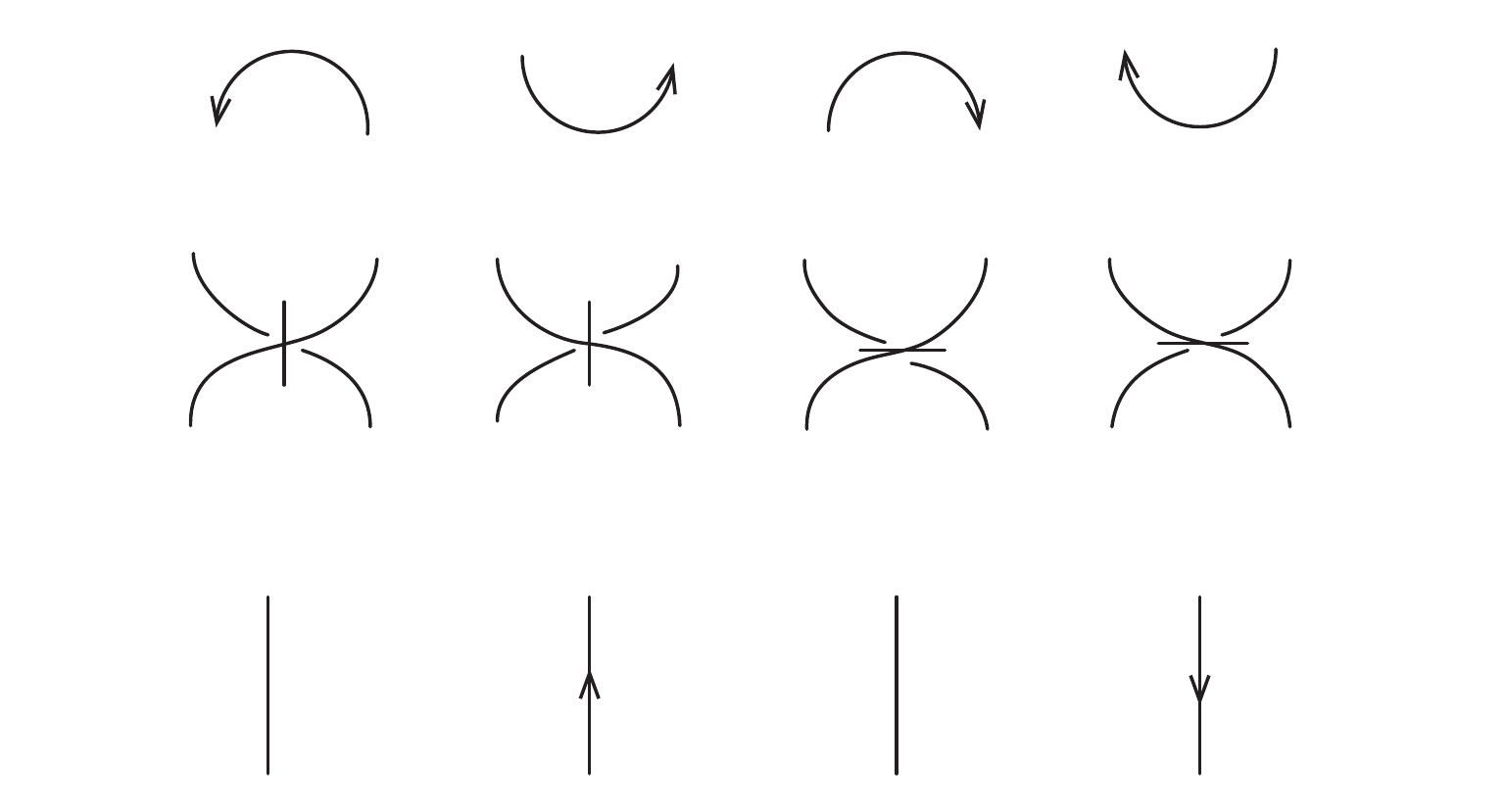}
\begin{picture}(0,0)
\put(-250,124){$n_{0 1}$ $=$ $A$}
\put(-195,124){$u_{0 1}$ $=$ $- A$}
\put(-135,124){$n_{1 0}$ $=$ $-A^{-1}$}
\put(-72,124){$u_{1 0}$ $=$ $A^{-1}$}
\put(-218,67){$j$}
\put(-253,67){$i$}
\put(-215,112){$j$}
\put(-253,112){$i$}
\put(-158,67){$j$}
\put(-193,67){$i$}
\put(-155,112){$j$}
\put(-193,112){$i$}
\put(-98,67){$j$}
\put(-133,67){$i$}
\put(-95,112){$l$}
\put(-133,112){$k$}
\put(-38,67){$j$}
\put(-73,67){$i$}
\put(-35,112){$l$}
\put(-73,112){$k$}
\put(-255,52){$A^{i j}_{i j}$ $=$ $A$}
\put(-205,52){$(A^{-1})^{i j}_{i j}$ $=$ $A^{-1}$}
\put(-121,52){$A^{k l}_{i j}$}
\put(-62,52){$B^{k l}_{i j}$}
\put(-233,22){$1$}
\put(-205,22){$=$}
\put(-110,22){$0$}
\put(-86,22){$=$}
\end{picture}
\caption{The indices $i$, $j$, $k$, $l$ are the element of $\{0, 1\}$.  The edges with indices $0$ and $1$ fix the orientations by using the rule on the last line in the figure.  We set $A^{0 1}_{0 1}$ $=$ $-A$, $A^{1 0}_{1 0}$ $=$ $- A^{-3}$, $A^{1 0}_{0 1}$ $=$ $A^{0 1}_{1 0}$ $=$ $A^{-1}$, $B^{0 1}_{0 1}$ $=$ $- A^{3}$, $B^{1 0}_{1 0}$ $=$ $- A^{-1}$ and, $B^{1 0}_{0 1}$ $=$ $B^{0 1}_{1 0}$ $=$ $A$.  }\label{bolz_w}
\end{figure}
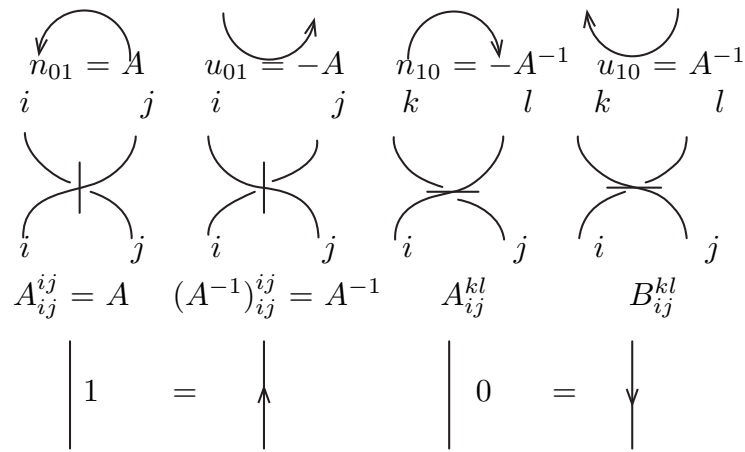
On the other hand, the $R$-matrix of the Kauffman bracket of links consists of $(A^{k l}_{i j})$ and $(B^{k l}_{i j})$ (cf. \cite{o}).  In fact, using Fig. \ref{bolz_w}, 
\begin{equation}
\left(
\begin{matrix}
A^{0 0}_{0 0} & A^{0 1}_{0 0} & A^{1 0}_{0 0} & A^{1 1}_{0 0}\\
A^{0 0}_{0 1} & A^{0 1}_{0 1} & A^{1 0}_{0 1} & A^{1 1}_{0 1}\\
A^{0 0}_{1 0} & A^{0 1}_{1 0} & A^{1 0}_{1 0} & A^{1 1}_{1 0}\\
A^{0 0}_{1 1} & A^{0 1}_{1 1} & A^{1 0}_{1 1} & A^{1 1}_{1 1}
\end{matrix}
\right) = \left(
\begin{matrix}
A & 0 & 0 & 0 \\
0 & 0 & A^{-1} & 0 \\
0 & A^{-1} & A - A^{-3} & 0 \\
0 & 0 & 0 & A
\end{matrix}
\right)
\end{equation}
and
\begin{equation}
\left(
\begin{matrix}
B^{0 0}_{0 0} & B^{0 1}_{0 0} & B^{1 0}_{0 0} & B^{1 1}_{0 0}\\
B^{0 0}_{0 1} & B^{0 1}_{0 1} & B^{1 0}_{0 1} & B^{1 1}_{0 1}\\
B^{0 0}_{1 0} & B^{0 1}_{1 0} & B^{1 0}_{1 0} & B^{1 1}_{1 0}\\
B^{0 0}_{1 1} & B^{0 1}_{1 1} & B^{1 0}_{1 1} & B^{1 1}_{1 1}
\end{matrix}
\right) = \left(
\begin{matrix}
A^{-1} & 0 & 0 & 0 \\
0 & A^{-1} - A^{3} & A & 0 \\
0 & A & 0 & 0 \\
0 & 0 & 0 & A^{-1}
\end{matrix}\right).  
\end{equation}
The matrix $(A^{k l}_{i j})$ is none other than the $R$-matrix corresponding to a crossing in the first figure from the left in Fig. \ref{crossing}, and $(B^{k l}_{i j})$ is the inverse of $(A^{k l}_{i j})$, corresponding to the second from the left in Fig. \ref{crossing}. 
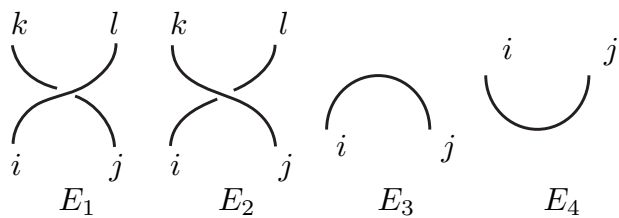
\begin{figure}
\qquad\qquad\qquad\qquad\quad
\includegraphics[width=7cm]{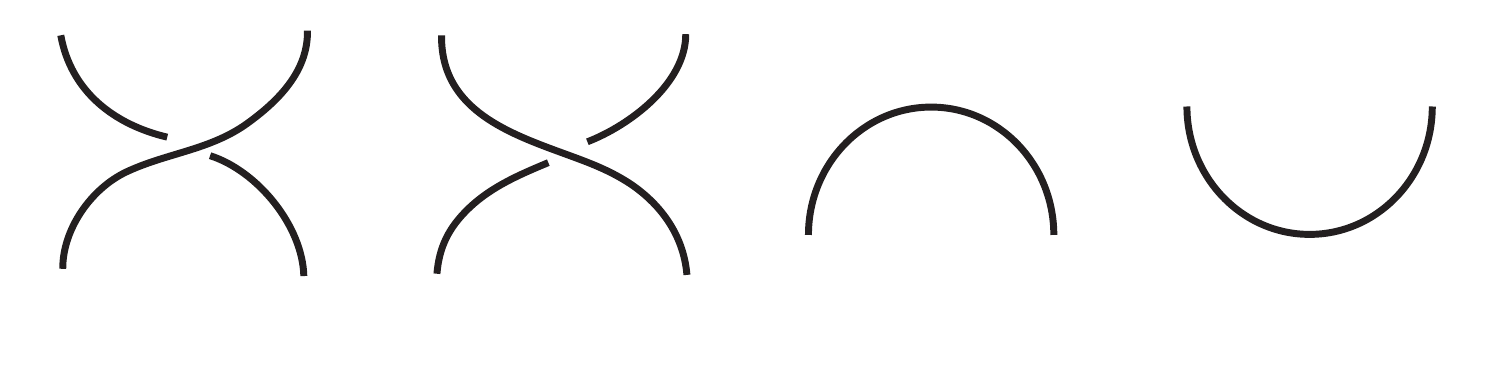}
\begin{picture}(0,0)
\put(-180,-5){$E_{1}$}
\put(-130,-5){$E_{2}$}
\put(-80,-5){$E_{3}$}
\put(-28,-5){$E_{4}$}
\put(-195,50){$k$}
\put(-164,50){$l$}
\put(-195,6){$i$}
\put(-164,6){$j$}
\put(-145,50){$k$}
\put(-111,50){$l$}
\put(-145,6){$i$}
\put(-111,6){$j$}
\put(-93,12){$i$}
\put(-60,12){$j$}
\put(-41,42){$i$}
\put(-9,42){$j$}
\end{picture}
\caption{The elementary part of diagrams with indices.  Each of the elementary part corresponds to the letter $E_n$ ($n$ $=$ $1$, $2$, $3$, or $4$) directly under the diagram.  }\label{crossing}
\end{figure}
Moreover, let us consider the matrices $n$ $=$ 
$\left(
\begin{matrix}
n_{0 0}, & n_{0 1}, & n_{1 0}, & n_{1 1}
\end{matrix}
\right)$ $=$ $\left(
\begin{matrix}
0, & A, & - A^{-1}, & 0 
\end{matrix}
\right)$ and $u$ $=$ 
$\left(
\begin{matrix}
0 \\ - A \\ A^{-1} \\ 0
\end{matrix}
\right)$.  The non-trivial elements of these matrices are also given in Fig. \ref{bolz_w}.  First, set $R$ $=$ $(R^{k l}_{i j})$ $=$ $(A^{k l}_{i j})$ and then, $R^{-1}$ $=$ $(B^{k l}_{i j})$.  Now we define the map $w :$ $\{E_n\}$ $\to$ $\mathbb{Z}[A, A^{-1}]$.  The $w(E_n)$ is $0$ if the orientations of $E_n$ fixed by the indices are not well-defined.  If the orientations of $E_n$ are well-defined, we define $w(E_n)$ as follows:  
\begin{equation}
w(E_1) = R^{k l}_{i j},\ w(E_2) = (R^{-1})^{k l}_{i j},\ w(E_3) = n_{i j},\ w(E_4) = u_{i j}.  
\end{equation}
\cite{v2} implies the following Theorem (also see \cite[Theorem 3.6]{o}).  
\begin{theorem}[Viro]
The elementary parts $E_n$ of a framed link diagram $D$ are defined by Fig. \ref{crossing}.  Let $S$ be the enhanced Kauffman state, as defined by the correspondence between the indices and orientations of the edges, and by applying the case of framed links to that of tangles.  For the Kauffman bracket $\langle D \rangle$ of a framed link diagram $D$, the following formula $(\ref{kauffman_b})$ is held $:$ 
\begin{equation}\label{kauffman_b}
\langle D \rangle = \sum_{S} \prod_{E_n} w(E_n).  
\end{equation}
\end{theorem}
\begin{remark}
The linear maps $R$, $R^{-1}$, $n$, and $u$ define the operator invariant of framed tangles \cite[Theorem 3.6]{o}.  
\end{remark}
\begin{remark}
The boundary operator sends an element of the diagonal matrix to an element of the other part of matrices $(A^{k l}_{i j})$ and $(B^{k l}_{i j})$ when we divide each matrix into two parts: 
\begin{equation}
(A^{k l}_{i j}) = A \left(
\begin{matrix}
1&0&0&0\\
0&1&0&0\\
0&0&1&0\\
0&0&0&1
\end{matrix}\right) + A^{-1}
\left(
\begin{matrix}
0&0&0&0\\
0&- A^{2}&1&0\\
0&1&- A^{-2}&0\\
0&0&0&0
\end{matrix}
\right)
\end{equation}
and
\begin{equation}
(B^{k l}_{i j}) = 
A^{-1}
\left(
\begin{matrix}
1&0&0&0\\
0&1&0&0\\
0&0&1&0\\
0&0&0&1
\end{matrix}
\right) + A
\left(
\begin{matrix}
0&0&0&0\\
0&- A^{2}&0&0\\
0&0&- A^{-2}&0\\
0&0&0&0
\end{matrix}
\right)
\end{equation}
\end{remark}
\begin{remark}
Let $E^{0}_{1}$ and $E^{1}_{2}$ (resp. $E^{1}_{1}$ and $E^{0}_{2}$) be $E_n$ with a positive (resp. negative) marker, as in Fig. \ref{bolz_w} and let $E_3$ $=$ $E^{0}_{3}$, $E_4$ $=$ $E^{0}_{4}$, and $I$ is the identity matrix.  Set $w(E^{0}_1)$ $=$ $A I$, $w(E^{1}_1)$ $=$ $(A^{k l}_{i j})$ $-$ $A I$, $w(E^{0}_{2})$ $=$ $A^{-1} I$, $w(E^{1}_2)$ $=$ $(B^{k l}_{i j})$ $-$ $A^{-1} I$, $w(E^{0}_{3})$ $=$ $n_{i j}$, and $w(E^{0}_{4})$ $=$ $u_{i j}$.  \cite{v2} implies that for an arbitrary framed link diagram $D$, 
\begin{equation}
\langle D \rangle = \sum_{S} \prod_{E^{m}_{n}} w(E^{m}_{n}).  
\end{equation}
\end{remark}

%%%%%%%%%%%%%%%%%%%%%%%%%%%%%%%%%%%%%%%%%%%%%%%%%%

\end{document}